\documentclass[11pt]{amsart}

\usepackage{amsmath,amsthm,amssymb}
\usepackage{mathrsfs} 

\usepackage[all]{xy}

\allowdisplaybreaks

\newtheorem{thm}{Theorem}[section]
\newtheorem{prop}[thm]{Proposition}

\newtheorem{cor}[thm]{Corollary}
\newtheorem{lem}[thm]{Lemma}

\theoremstyle{definition}

\numberwithin{equation}{section}
\newtheorem{rem}[thm]{\bf Remark}
\newtheorem{ex}[thm]{\bf Example}

\newtheorem{defn}[thm]{\bf Definition}

\newtheorem{const}[thm]{\bf Construction}

\def\bbQ{\mathbb{Q}}
\def\bbR{\mathbb{R}}

\def\bbZ{\mathbb{Z}}

\def\bfG{\mathbf{G}}

\def\bfSigma{\mathbf{\Sigma}}

\def\bfc{\mathbf{c}}

\def\bfe{\mathbf{e}}

\def\bfp{\mathbf{p}}

\def\bfu{\mathbf{u}}

\def\bfw{\mathbf{w}}
\def\bfx{\mathbf{x}}
\def\bfy{\mathbf{y}}

\def\rmid{\mathrm{id}}

\def\rmpr{\mathrm{pr}}

\def\calF{\mathcal{A}}

\def\calC{\mathcal{C}}
\def\calF{\mathcal{F}}

\def\calP{\mathcal{P}}

\def\frakd{\mathfrak{d}}
\def\frakD{\mathfrak{D}}
\def\frakDs{\mathfrak{D}_{\fraks}}
\def\frakg{\mathfrak{g}}

\def\frakj{\mathfrak{j}}
\def\frakp{\mathfrak{p}}

\def\fraks{\mathfrak{s}}
\def\scrL{\mathscr{L}}
\def\scrH{\mathscr{H}}

\def\l{\ell}
\def\d{\delta}
\def\ss{\scriptstyle}

\begin{document}

\title[Dilogarithm identities in CSD]{
Dilogarithm identities\\ in cluster scattering diagrams
\ \\
}
\address{\noindent Graduate School of Mathematics, Nagoya University, 
Chikusa-ku, Nagoya,
464-8604, Japan}
\email{nakanisi@math.nagoya-u.ac.jp}
\author[T. Nakanishi]{Tomoki Nakanishi}
\maketitle

\begin{abstract}
We extend the notion of  $y$-variables (coefficients) in cluster algebras to cluster scattering diagrams. Accordingly, we extend the dilogarithm identity associated with a period in a cluster pattern to the one associated with a loop in a cluster scattering diagram. We show that  these identities are constructed from  and reduced  to trivial ones by applying the pentagon identity possibly infinitely many times.
\end{abstract}

\section{Introduction}

In the seminal paper by Gross-Hacking-Keel-Kontsevich \cite{Gross14},
 it was revealed that there is an intimate
relation between cluster algebras and scattering diagrams. 
The underlying cluster pattern for a cluster algebra is constructed by \emph{mutations} \cite{Fomin02},
while a scattering diagram is  constructed based on the \emph{consistency} \cite{Kontsevich06, Gross07}.
In spite of the  difference of their principles, the whole information of a given cluster pattern
$\bfSigma$
is contained in a certain scattering diagram $\frakD$ called
a \emph{cluster scattering diagram} (CSD, for short).

Let us briefly summarize the relation between $\bfSigma$ and  $\frakD$. 
Let $\Delta(\bfG^{t_0})$ be the fan  (the \emph{$G$-fan})  spanned by the cones of $G$-matrices
of $\bfSigma$ with the initial vertex $t_0$.
The $G$-fan $\Delta(\bfG^{t_0})$ is isomorphic to the \emph{cluster complex}  $\Delta(\bfSigma)$ of $\bfSigma$ in \cite{Fomin03a}
as a simplicial complex.
Then, 
the union of the codimension one cones of
$\Delta(\bfG^{t_0})$ is embedded into the support $\mathrm{Supp}(\frakD)$ of $\frakD$.
Also, it coincides with  $\mathrm{Supp}(\frakD)$
if and only if $\bfSigma$ is of finite type.
In general,
the structure of $\frakD$ outside
 the support 
of the $G$-fan $\Delta(\bfG^{t_0})$ is very complicated.

So, at least superficially, a cluster pattern carries only a partial information of $\frakD$.
However, looking at the construction of  $\frakD$ in
 \cite{Gross14}, 
 we recognize that   $\frakD$ is  uniquely  constructed (up to isomorphism) only from
 the \emph{initial exchange matrix} $B_{t_0}$ of $\bfSigma$.
 Thus, we may think that a cluster pattern $\bfSigma$
 and a CSD $\frakD$ are actually \emph{one inseparable object}
 associated with $B_{t_0}$.

There is one great advantage of working with $\frakD$ rather than only $\bfSigma$;
namely, it has an underlying group $G$ which we call the \emph{structure group} for $\frakD$.
Since $G$ is defined by the inverse limit,  the infinite product is well-defined.
Then, first of all, it is possible to consider an \emph{infinite product of mutations},
which is nothing but a \emph{path-ordered product} in $\frakD$.
Moreover, the group $G$ controls the whole structure of $\frakD$ locally and globally.
In particular,
the following facts were presented somewhat implicitly in \cite{Gross14},
and described more explicitly in \cite{Nakanishi22a}:
\begin{itemize}
\item[(a).]
There are some distinguished elements  of $G$ called
 the \emph{dilogarithm elements}, and they satisfy  the \emph{pentagon relation}.
\item[(b).]
Up to equivalence, $\frakD$ consists of walls whose wall elements
are dilogarithm elements with positive rational exponents.
\item[(c).]
Dilogarithm elements act as \emph{mutations}
on cluster variables in $\bfSigma$.
\item[(d).]
Every \emph{consistency relation} in  $\frakD$ is constructed from and reduced to a trivial one
(i.e., $g=g$ with $g\in G$) by applying the pentagon relation (and commutative relations) possibly infinitely many times.
\end{itemize}
Thus, the dilogarithm elements and the pentagon relation
are everything for a CSD; moreover,  they build a bridge over the gap between the principles
of $\bfSigma$ and $\frakD$, namely, mutations and consistency.
We note that
the importance of the dilogarithm  in scattering diagrams was already noticed in \cite{Kontsevich08}. Also, the construction of scattering diagrams by the pentagon relation appeared in \cite{Gross09} in a geometrical setting.

With this perspective in mind,
it is natural to extend basic notions for a cluster pattern to a CSD.
The notion of a \emph{seed}, which is most fundamental for a cluster pattern,
 cannot be extended to the entire CSD,
because there is no global chamber structure in a CSD in general \cite{Gross14}.
However, the notions of cluster variables ($x$-variables),
coefficients ($y$-variables), and their mutations can be extended to the entire CSD.
In fact, the \emph{theta functions} in \cite{Gross14} is an extension of
the notion of \emph{cluster monomials}.
In this paper we  extend the notion of  \emph{$y$-variables}  to a CSD.
Accordingly, we  also extend the \emph{dilogarithm identity} (DI, for short) associated with a 
\emph{period} in a cluster pattern \cite{Nakanishi10c}
 to  the one associated with a \emph{loop} in a CSD
 based on the classical mechanical method in \cite{Gekhtman16}.
This is our first main result (Theorem~\ref{thm:DI1}).

Even though the structure of a CSD is  very  complicated as already mentioned,
it is tractable in some sense
thanks to the property (d).
As a result,  the above DIs 
are also constructed from and
reduced to  trivial ones by applying the celebrated  \emph{pentagon identity} (the \emph{five-term relation}) of the dilogarithm function
\cite{Lewin81}
\emph{possibly infinitely many times} in a parallel way.
(See \cite[\S2.A]{Zagier07} for the background of the reducibility problem of the dilogarithm identities
to the pentagon identity.)
This also provides an alternative   and independent proof of the DIs.
This is our second main result (Theorem~\ref{thm:DI2}).
We also give  examples of  DIs for rank 2 CSDs of affine type.

We conclude by giving some remarks on related works.

(i). The $y$-variables are also known as \emph{cluster $\mathcal X$-variables} \cite{Fock03}.
The corresponding theta functions were studied via the CSDs with principal coefficients in \cite{Gross14} and via the scattering diagrams with the Langlands dual data
($\mathcal X$ scattering diagrmas) in \cite{Cheung19, Cheung21}.
Apparently our    $y$-variables in CSDs are defined in a different way,
and we do not know how
they   are related with the ones in the above works. See also Remark \ref{rem:yv1} (c).
The results in this paper show that our $y$-variables behave nicely at least in view of
dilogarithm identities.

(ii).
The \emph{quantum analogue} of a CSD (QCSD) was constructed and studied by 
\cite{Mandel15,Davison19,Cheung20} in the skew-symmetric case,
and it is naturally extended to the skew-symmetrizable case as well.
From their construction,
we immediately obtain the \emph{quantum DI (QDI)}
(possibly with infinite product)
associated with a loop in a  QCSD for the quantum dilogarithm of  \cite{Faddeev93,Faddeev94},
which are  the extension of the  QDIs for a cluster pattern
studied in \cite{Fock07,Keller11,Nagao11b, Kashaev11}.
They are clearly the counterpart of the DIs in this paper.
However, we stress that the main results in this paper
are \emph{not} straightforwardly obtained from the quantum case
by two reasons.
Firstly, the DIs here should be obtained from QDI
in their semiclassical limits in the sense of 
\cite{Kashaev11}; however, the method therein is only heuristic and not rigorous
even for the finite product case.
Secondly, it is not yet known that every constancy relation in a QCSD
is reduced to a trivial one by the pentagon relation.

\medskip
Acknowledgments.
The author thanks  Bernhard Keller for suggesting the problem
of extending the dilogarithm identities from cluster patterns to CSDs.
He also thanks referees for helpful comments.
This work is supported in part by JSPS Grant No.~JP16H03922.

\section{Cluster scattering diagrams}
In this section we quickly recall  basic notions for CSDs in \cite{Gross14}.
See also \cite{Nakanishi22a} for a review, which is closer to the present context.

\subsection{Structure group}
Let $\Gamma$ be a \emph{fixed data} consisting of the following:
\begin{itemize}
\item a lattice $N$ of rank $r$,
\item a skew-symmetric bilinear form $\{\cdot, \cdot\}: N \times N \rightarrow \bbQ$,
\item a sublattice $N^{\circ}\subset N$ of finite index such that  $\{N^{\circ}, N\}\subset \bbZ$,
\item positive integers $\d_1$, \dots, $\d_r$ such that
there is a basis $(e_1,\dots,e_r)$ of $N$, where $(\d_1e_1,\dots,\d_re_r)$
is a basis of $N^{\circ}$,
\item $M=\mathrm{Hom}(N, \bbZ)$ and $M^{\circ}=\mathrm{Hom}(N^{\circ}, \bbZ)$.
\end{itemize}
Let $M_{\bbR}=M\otimes_{\bbZ} \bbR$.
Let
$\langle n, m\rangle$ denote  the canonical paring either for $N^{\circ}\times M^{\circ}$
or for $N\times M_{\bbR}$.
For $n\in N$, $n\neq 0$, let $n^{\perp}:=\{ z\in M_{\bbR} \mid \langle n, z\rangle =0\}$.

Let $\fraks=(e_1,\dots, e_r)$ be a \emph{seed} for $\Gamma$,
which are a basis of $N$ such that $(\d_1e_1,\dots, \d_re_r)$ are a basis of $N^{\circ}$.
Let 
\begin{align}
N^+=\Biggl\{ \sum_{i=1}^r a_i e_i \mid a_i \in \bbZ_{\geq 0}, \sum_{i=1}^r a_i >0\Biggr\}
\end{align}
be the set of the positive vectors of $N$ with respect to $\fraks$.
Let $N^+_{\rmpr}$ denote the set of primitive elements in $N^+$.
The degree function $\deg: N^+\rightarrow \bbZ_{> 0}$ is defined by $\deg( \sum_{i=1}^r a_i e_i ):=\sum_{i=1}^r a_i$.
Let $\frakg$ be the $N^+$-graded Lie algebra defined by
\begin{align}
\label{eq:Xcom1}
\frakg=\bigoplus_{n\in N^+} \bbQ X_n,
\quad
[X_n, X_{n'}]:=\{n,n'\}X_{n+n'}.
\end{align}
Let $\widehat \frakg$ be the completion of $\frakg$ with respect to $\deg$, and let $G=G_{\Gamma,\fraks}=\exp(\widehat \frakg)$
be the exponential group of $\widehat \frakg$.
Namely,
an element of $G$ is given by a formal symbol $\exp(X)$ ($X\in  \widehat \frakg$)
and the product is defined by the Baker-Campbell-Hausdorff formula.
We call $G$ the \emph{structure group} for the forthcoming scattering diagrams.
Let $G^{>\ell}$ be the normal subgroup of $G$ generated by the elements with degrees greater than $\ell$,
and let $G^{\leq \ell}:=G/G^{>\ell}$ be its quotient.
For each $n\in N_{\rmpr}^+$, let $G_n^{\parallel}$ be the abelian subgroup of $G$ 
generated by elements $\exp(c_{jn}X_{jn})$ ($j>0$, $c_{jn}\in \bbQ$)
admitting the infinite product.

\subsection{Initial seed for cluster pattern}
Let $(e^*_1, \dots, e^*_r)$ be the basis of $M$ which is dual to $\fraks=(e_1,\dots, e_r)$.
Let $f_i= \d_i^{-1}e^*_i$.
Then,  $(f_1, \dots, f_r)$ is a basis of $M^{\circ}$, which is dual to
the basis $(\d_1 e_1,\dots, \d_r e_r)$ of $N^{\circ}$.
Let $\bfx=(x_i)_{i=1}^r$, $\bfy=(y_i)_{i=1}^r$, $B=B_{\Gamma,\fraks}=(b_{ij})_{i,j=1}^r$ with
\begin{align}
\label{eq:xyB1}
x_i = x^{f_i}, \quad y_i = y^{e_i}, \quad b_{ij}=\{\d_i e_i,e_j\},
\end{align}
where $x$ and $y$ are formal variables for formal exponentials.
Then, $B$ is a skew-symmetrizable integer matrix with
a  left skew-symmetrizer $\Delta^{-1}$, where we set
 $\Delta=\mathrm{diag}(\delta_1,\dots,\delta_r)$.
Thus, $\Sigma=(\bfx, \bfy, B)$ is naturally identified with
the \emph{initial} seed for the cluster pattern $\bfSigma_{\fraks}$ corresponding to $\frakD_{\fraks}$.
Conversely, for a given skew-symmetrizable integer matrix $B$, there is a (not unique) pair of $\Gamma$ and $\fraks$ such that $B=B_{\Gamma,\fraks}$.
Moreover,
two groups $G_{\Gamma,\fraks}$ and $G_{\Gamma',\fraks'}$ are isomorphic
if $B_{\Gamma,\fraks}=B_{\Gamma',\fraks'}$.
In this paper, we concentrate on $y$-variables only.

\subsection{$y$-representation}
We introduce a monoid $Q:=N^+ \sqcup \{0\}\subset N$,
the monoid algebra $\bbQ[Q]$ of $Q$ over $\bbQ$,
and its completion $ \bbQ[[Q]]$ with respect to $\deg$.
An element of $ \bbQ[[Q]]$  is written as a formal infinite sum with a formal variable $y$ as
\begin{align}
\label{eq:y0}
\sum_{n\in Q}c_n y^{n},
\quad
(c_n\in \bbQ).
\end{align}
It is also regarded as a formal power series of the
\emph{initial $y$-variables} $\bfy=(y_1,\dots,y_r)$
in \eqref{eq:xyB1}.
We consider an action of $X_n$ ($n\in N^+$)  on $\bbQ[[Q]]$
\begin{align}
\label{eq:Xnaction1}
X_n(y^{n'})= \{n,n'\} y^{n'+n},
\end{align}
which is a derivation,
and we linearly extend it to the action of any $X\in \widehat\frakg$.
Then, the resulting map $\widehat\frakg \rightarrow \mathrm{Der}(\bbQ[[Q]])$
is a Lie algebra homomorphism.
Moreover, its exponential action
\begin{align}
(\mathrm{Exp}\, X)(y^{n'}):=
\sum_{j=0}^{\infty}
\frac{1}{j!} X^j(y^{n'})
\quad
(X\in \widehat\frakg)
\end{align}
 is well-defined on $\bbQ[[Q]]$.
 Then, $\mathrm{Exp}\, X$
 is an algebra automorphism of $\bbQ[[Q]]$ 
(e.g., \cite[\S1.2]{Jacobson79}).
Thus, we have a representation of $G$
\begin{align}
\begin{matrix}
\rho_y:&  G& \rightarrow & \mathrm{Aut}(\bbQ[[Q]])
\\
& \exp X & \mapsto &\mathrm{Exp}\, X.
\end{matrix}
\end{align}
We call it the \emph{$y$-representation} of $G$.

\subsection{Dilogarithm elements and pentagon relation}

Let us introduce some distinguished elements in $G$.
\begin{defn}[Dilogarithm element]
For any $n\in N^+$, we define
\begin{align}
\label{eq:diloge1}
\Psi[n]:=\exp\Biggl(\,  \sum_{j=1}^{\infty} \frac{(-1)^{j+1}}{j^2} X_{jn}\Biggr) \in G_{n_0}^{\parallel},
\end{align}
where $n_0\in  N_{\rmpr}^+$ is the one such that $n=jn_0$ for some $j\in \bbZ_{>0}$.
We call it the \emph{dilogarithm element} for $n$.
\end{defn}

The element  $\Psi[n]$ acts on  $ \bbQ[[Q]]$ under $\rho_y$ as
\begin{align}
\label{eq:ymut1}
\begin{split}
\Psi[n](y^{n'})&= \exp\Biggl(\,\sum_{j=1}^{\infty} \frac{(-1)^{j+1}}{j^2} X_{jn}\Biggr)(y^{n'})\\
&= y^{n'}\exp\Biggl(\{n, n'\} \sum_{j=1}^{\infty} \frac{(-1)^{j+1}}{j} y^{jn}\Biggr)\\
 &=y^{n'}(1+y^n)^{\{n,n'\}}.
 \end{split}
\end{align}
Observe that
this is essentially the automorphism part of the Fock-Goncharov decomposition
of mutations of $y$-variables \cite[\S 2.1]{Fock03}.

It is easy to see that
the dilogarithm elements $\Psi[n]^c$ ($n\in N^+$, $c\in \bbQ$) are generators of $G$
admitting the infinite product.
Moreover, they satisfy the following remarkable relations.
\begin{prop}
\label{prop:pent1}
Let $n_1,n_2\in N^+$.
The following relations hold in $G$.
\par
(a). If $\{n_2,n_1\}=0$,
for any $c_1,c_2\in \bbQ$,
\begin{align}
\label{eq:com1}
\Psi[n_2]^{c_2} \Psi[ n_1]^{c_1}=\Psi[  n_1 ]^{c_1} \Psi[ n_2]^{c_2}.
\end{align}

(b). {\rm (Pentagon relation
\cite[Example 1.14]{Gross14}, \cite[Prop.\ III.1.13]{Nakanishi22a}.)} If $\{n_2,n_1\}=c$ $(c\in \bbQ, c\neq 0)$,
\begin{align}
\label{eq:pent1}
\quad
\Psi[n_2 ]^{1/c} \Psi[ n_1]^{1/c}=
\Psi[  n_1 ]^{1/c} \Psi[n_1+n_2]^{1/c} \Psi[ n_2]^{1/c}.
\end{align}
\end{prop}
\begin{proof}
The equality \eqref{eq:com1} is clear by \eqref{eq:ymut1}.
One can prove \eqref{eq:pent1} by comparing the  actions of both  sides on $y^{n'}$
using \eqref{eq:ymut1}.
\end{proof}

\subsection{Cluster scattering diagrams}
A wall $\bfw=(\frakd, g)_{n}$ for $\fraks$ is a triplet such that
$n\in N_{\rmpr}^+$, a cone $\frakd\subset n^{\perp}$ of codimension 1, and $g\in G_n^{\parallel}$. 
We call $n$, $\frakd$, $g$, the \emph{normal vector}, the \emph{support}, the \emph{wall element}
of $\bfw$, respectively.
Let $p^*:N \rightarrow M^{\circ}$, $n \mapsto \{ \cdot, n\}$.
We say that a wall $\bfw=(\frakd, g)_{n}$ is \emph{incoming} if $p^*(n)\in \frakd$ holds. 

\begin{defn}[Scattering diagram]
A \emph{scattering diagram} $\frakD=\{ \bfw_{\lambda}=(\frakd_{\lambda}, g_{{\lambda}})_{n_{\lambda}}
\}_{\lambda\in \Lambda}$ for $\fraks$ is a collection of walls for $\fraks$
satisfying the  following finiteness condition:
\emph{For any degree $\ell$,
there are only finitely many walls such that $\pi_{\ell}(g_{\lambda})\neq \rmid$,
where $\pi_{\ell}:G \rightarrow G^{\leq \ell}$ is the canonical projection.}
\end{defn}

For a scattering diagram  $\frakD$, we define
\begin{align}
\mathrm{Supp}(\frakD)&=\bigcup_{\lambda\in \Lambda} \frakd_{\lambda},
\quad
\mathrm{Sing}(\frakD)=\bigcup_{\lambda\in \Lambda} \partial\frakd_{\lambda}
\cup
\bigcup_{\ss \lambda, \lambda'\in \Lambda \atop \ss \dim 
\frakd_{\lambda}\cap \frakd_{\lambda'}=r-2
} \frakd_{\lambda}\cap \frakd_{\lambda'}.
\end{align}

A curve $\gamma:[0,1]\rightarrow M_{\bbR}$ is
\emph{admissible} for  $\frakD$
if it satisfies the following properties:
\begin{itemize}
\item[(i)]
The end points of $\gamma$ are in $M_{\bbR}\setminus \mathrm{Supp}(\frakD)$.
\item[(ii)]
It is a smooth curve, and it intersects $\mathrm{Supp}(\frakD)$ transversally.
\item[(iii)]
$\gamma$ does not intersect  $\mathrm{Sing}(\frakD)$.
\end{itemize}
For any admissible curve $\gamma$,
the path-ordered product $\frakp_{\gamma, \frakD}\in G$ is defined
as the product of
the wall elements $g_{{\lambda}}^{\epsilon_{\lambda}}$ of walls $\bfw_{\lambda}$ of $\frakD$ intersected by $\gamma$
in the order of intersection, where $\epsilon_{\lambda}$ is the \emph{intersection sign}
defined by
\begin{align}
\label{eq:int1}
\epsilon_{\lambda}=
\begin{cases}
1 & \langle n_{\lambda}, \gamma'\rangle <0,\\
-1 & \langle n_{\lambda}, \gamma'\rangle >0,
\end{cases}
\end{align}
and $\gamma'$ is the velocity vector of $\gamma$ at the wall $\bfw_{\lambda}$.
The product $\frakp_{\gamma, \frakD}$ is an infinite one in general,
and it is well-defined in $G$ due to the finiteness condition.
See \cite[\S1.1]{Gross14} for a more precise definition.
We say that a pair of scattering diagrams $\frakD$  and $\frakD'$ are \emph{equivalent} if  $\frakp_{\gamma, \frakD}=\frakp_{\gamma, \frakD'}$
for any admissible curve $\gamma$  for both $\frakD$ and $\frakD'$.
We say that a scattering diagram $\frakD$ is \emph{consistent} if  $\frakp_{\gamma, \frakD}=\rmid$
for any admissible loop (i.e., closed curve) $\gamma$ for $\frakD$.

\begin{defn}[Cluster scattering diagram]
A \emph{cluster scattering diagram} $\frakD_{\fraks}$ (CSD, for short) for $\fraks$ is a consistent scattering diagram 
whose set of incoming walls are given by
\begin{align}
\mathrm{In}_{\fraks}:=\{ \bfw_{e_i}=(e_i^{\perp}, \Psi[e_i]^{\d_i})_{e_i} \mid i=1,\dots, r\}.
\end{align}
\end{defn}

For $n\in N^+$,  let $\d(n)$ be the smallest positive rational number 
such that $\d(n) n \in N^{\circ}$,
which is called the \emph{normalization factor} of $n$.
For example, $\d(e_i)=\d_i$.
Note that $\d(tn)=\d(n)/t$ holds for any $n\in N^+_{\rmpr}$ and $t\in \bbZ_{>0}$.
Also, $\d(n)$ is an integer for any $n\in N^+_{\rmpr}$,

The following is the most fundamental theorem on CSDs.

\begin{thm}[{\cite[Theorems 1.12 \& 1.13]{Gross14}}]
\label{thm:CSD1}
(a). There exists a CSD $\frakD_{\fraks}$ uniquely up to  equivalence.
\par
(b).
There exists a (still not unique) CSD $\frakD_{\fraks}$ such that
every wall element   has
the form
\begin{align}
\label{eq:gt1}
\Psi[tn]^{s \d(tn)}
\quad (n\in N_{\rmpr}^+,\ t, s\in \bbZ_{> 0}).
\end{align}
\end{thm}
In this paper we exclusively use $\frakDs$ given in Theorem \ref{thm:CSD1} (b),
which we call a \emph{positive realization} a CSD $\frakDs$.

\begin{rem}
The fact $s\in \bbZ_{> 0}$ is the key to prove the \emph{Laurent positivity} of cluster variables and
theta functions
in \cite{Gross14}, though we do not use this connection in this paper.
\end{rem}

\section{Dilogarithm identities in CSDs}

\subsection{$y$-variables for CSD}
Let us extend the  notion of  $y$-variables (coefficients) for a cluster pattern $\bfSigma_{\fraks}$ to a CSD $\frakD_{\fraks}$.
We say that a curve is \emph{weakly admissible} for $\frakD$ if it satisfies the
conditions (ii) and (iii) for an admissible curve.
The definition of the path-ordered product $\frakp_{\gamma,\frakD}$ is extended  to a weakly admissible curve $\gamma$
by ignoring the contribution from the walls at the end points.

\begin{defn}[$y$-variable/$c$-vector for CSD]
Let $\bfw=(\frakd, g)_{n}$ be any wall of $\frakD_{\fraks}$
with $g=\Psi[tn]^{s \d(tn)}$ as in  \eqref{eq:gt1}.
Let $z\in \frakd$
with $z\not\in \mathrm{Sing}(\frakD_{\fraks})$.
Let 
\begin{align}
\calC^+:=\{ z \in M_{\bbR}\mid \langle e_i, z\rangle \geq 0\  (i=1,\dots, r)\}.
\end{align}
Let $\gamma_{z}$ be any weakly admissible curve in $\frakD_{\fraks}$
 from $z$ to any point in $\mathrm{Int}(\calC^+)$.
Then, we define a \emph{$y$-variable $y_{z}[tn]$ at $z$ with  the $c$-vector} $tn$ by
\begin{align}
\label{eq:y1}
y_{z}[tn]:=\frakp_{\gamma_{z},\frakD_{\fraks}}(y^{tn}) \in \bbQ[[Q]],
\end{align}
where the path-ordered product $\frakp_{\gamma_{z},\frakD_{\fraks}}\in G$ acts on  $ y^{tn}$
under the $y$-represent\-ation $\rho_y$.
\end{defn}

Since any $g\in G_n^{\parallel}$ acts trivially on $y^{tn}$, $y_{z}[tn]$ is independent of the choice of
$\gamma_z$ due to the consistency of $\frakD_{\fraks}$.
Also, due to our assumption on $\frakDs$, any wall element of $\frakD_{\fraks}$ has   the form \eqref{eq:gt1}.
Therefore,
$\frakp_{\gamma_{z},\frakD_{\fraks}}$ acts as a (possibly infinite) product
of   mutations in \eqref{eq:ymut1}.

\begin{rem}
\label{rem:yv1}
(a).
If $z$ belongs to a codimension 1 face of a  cluster chamber ($G$-cone) of $\frakDs$,
we have
 $t=s=1$ by the mutation invariance of $\frakDs$ \cite{Gross14}.
Moreover, $n=\varepsilon_{i;t} \bfc_{i;t}$ and
 $y_{z}[n]=y_{i;t}^{\varepsilon_{i;t}}$,
  where $y_{i;t}$ is an ordinary $y$-variable,
  and     $\varepsilon_{i;t}$ and $\bfc_{i;t}$ are the tropical sign
   and the $c$-vector for $y_{i;t}$, respectively.
The notion of a \emph{seed} cannot  be entirely extended for $\frakDs$,
because there is no overall chamber structure therein.
Thus, the composite mutation  \eqref{eq:y1} directly connects 
 the initial $y$-variables $\bfy$ and a given \emph{single} $y$-variable $y_{z}[tn]$
for $\frakDs$.
\par
(b).
Unlike  usual $y$-variables for a cluster pattern, 
all $y$-variables here have \emph{positive} $c$-vectors.
\par
(c).
Let $\theta_{Q,m}$ be the \emph{theta functions} in \cite{Gross14},
where $Q\in \mathrm{Int}(\calC^+)$ and
 $m \in \bbZ^r$.
 When $m$ belongs to a cluster chamber,
 the following formula holds  \cite[Theorem 4.9]{Gross14}:
 \begin{align}
\label{eq:theta1}
\theta_{Q,m}=\frakp_{\gamma,\frakD_{\fraks}}(x^{m}),
\end{align}
 where 
 the action of $\frakp_{\gamma,\frakD_{\fraks}}$ on $x^m$ is given by the \emph{(principal) $x$-representation} \cite[\S III.4]{Nakanishi22a}.
 Clearly, the definition \eqref{eq:y1}  is parallel to the formula \eqref{eq:theta1}.
 On the other hand, the formula \eqref{eq:theta1} is not valid  for general $m$.
 Thus, the relation between our $y$-variables and the theta functions
is not clear in general.

\par
(d). The variable $y_{z}[tn]$ changes discontinuously 
when $z\in \frakd$ crosses  the codimension 2 intersection (a joint) $\frakd\cap \frakd'$  with the support $\frakd'$
of another wall in $\frakDs$,
because $\frakp_{\gamma_{z},\frakD_{\fraks}}$ changes.
\end{rem}

\subsection{Euler and Rogers dilogarithms}

We define the Euler dilogarithm \cite{Lewin81} 
\begin{align}
\mathrm{Li}_2(x)&= \sum_{j=1}^{\infty}\frac{1}{j^2}x^j
\quad
(|x|< 1)
\\
&=-\int_{0}^{x} \frac{\log (1-y)}{y} \, dy
\quad
(x\leq 1),
\end{align}
and the Rogers dilogarithm
\begin{align}
L(x)
=
\mathrm{Li}_2(x)+ \frac{1}{2}\log x \log(1-x)
\quad
(0\leq x\leq 1).
\end{align}
The celebrated \emph{pentagon identity} (Abel's identity, the five-term relation) for $\mathrm{Li}_2(x)$ is neatly expressed
in terms of $L(x)$ as $(0\leq x, y <1)$
\begin{align}
\label{eq:pentL0}
L(x)+L(y)
=L(xy) + L\biggl( \frac{x(1-y)}{1-xy}\biggr)
+ L\biggl( \frac{y(1-x)}{1-xy}\biggr).
\end{align}

We introduce a variant of the Rogers dilogarithm,
which we call the \emph{modified Rogers dilogarithm}, as follows:
\begin{align}
\tilde L(x):&=
L\left(\frac{x}{1+x}\right)= -  \mathrm{Li}_2(-x) - \frac{1}{2}\log x \log(1+x)
\\
\label{eq:modL1}
&=
\frac{1}{2} \int_{0}^x 
\left\{ \frac{\log (1+y)}{y}
-
\frac{\log y}{1+y}
\right\}
dy
\quad
(0 \leq x).
\end{align}
The function
$\tilde L(x)$ is smooth but not analytic at $x=0$; however, it has the following  
Puiseux
 expansion
around $x=0$ \emph{with $\log$ factor}:
\begin{align}
\label{eq:Lexp1}
\tilde L(x)
&=
 \sum_{j=1}^{\infty} \frac{(-1)^{j+1} }{j^2}x^j 
-\frac{1}{2}\log x \sum_{j=1}^{\infty} \frac{(-1)^{j+1} }{j}x^j .
\end{align}
The pentagon identity \eqref{eq:pentL0} is expressed in terms of $\tilde L(x)$
as follows (e.g., \cite[\S 5.3]{Nakanishi11c}):
\begin{align}
\label{eq:pentL1}
\begin{split}
&\tilde L (y_2(1+y_1)) + \tilde L (y_1)
\\
&\qquad =
\tilde L(y_1(1+y_2+y_1y_2)^{-1})
+
\tilde L(y_1y_2(1+y_2)^{-1})
+
\tilde L(y_2).
\end{split}
\end{align}
The above arguments in $\tilde L(x)$ are identified with 
the $y$-variables with positive $c$-vectors in the $Y$-pattern  of type $A_2$ with
the initial exchange matrix
\begin{align}
B=
\begin{pmatrix}
0 & -1\\
1 & 0
\end{pmatrix}
\end{align}
as follows:
\begin{align}
\label{eq:A21}
\begin{split}
&
\begin{pmatrix}
\framebox{$y_1$}\\
y_2
\end{pmatrix}
\buildrel   {\mu_1} \over {\rightarrow}
\begin{pmatrix}
y_1^{-1}\\
\framebox{$y_2(1+y_1)$}
\end{pmatrix}
\buildrel   {\mu_2} \over {\rightarrow}
\begin{pmatrix}
y_1^{-1}(1+y_2+y_1y_2)\\
y_2^{-1}(1+y_1)^{-1}
\end{pmatrix},
\\
&
\begin{pmatrix}
y_1\\
\framebox{$y_2$}
\end{pmatrix}
\buildrel   {\mu_2} \over {\rightarrow}
\begin{pmatrix}
\framebox{$y_1y_2(1+y_2)^{-1}$}\\
y_2^{-1}
\end{pmatrix}
\buildrel   {\mu_1} \over {\rightarrow}
\begin{pmatrix}
{y_1^{-1}y_2^{-1}(1+y_2)}\\
\framebox{$y_1(1+y_2+y_1y_2)^{-1}$}
\end{pmatrix}
\\
&\hskip170pt
\buildrel   {\mu_2} \over {\rightarrow}
\begin{pmatrix}
y_2^{-1}(1+y_1)^{-1}\\
y_1^{-1}(1+y_2+y_1y_2)
\end{pmatrix}.
\end{split}
\end{align}
Also, observe the \emph{pentagon periodicity} \cite[\S2]{Fomin02}.
Namely, both results in \eqref{eq:A21} coincide up to
the transposition $\tau_{12}$.

\begin{rem}
A dilogarithm element  $\Psi[n]$ in \eqref{eq:diloge1}
is ``formally'' expressed as
\begin{align}
\Psi[n]=\exp( -\mathrm{Li}_2(-X_n))
\end{align}
if  we  ``identify'' the power $X_n^j$ with $X_{jn}$.
Be careful, however, that this is only formal, because $X_n^j$ and $X_{jn}$ act differently on $\bbQ[[Q]]$.
\end{rem}

\subsection{Dilogarithm identity associated with a loop in $\frakDs$}
Let $\frakD_{\fraks,\ell}$ be the scattering diagram consisting of walls  of $\frakD_{\fraks}$
such that $\pi_{\ell}(g_{\lambda})\neq \rmid$ in $G^{\leq  \ell}$.
We call $\frakD_{\fraks,\ell}$ the \emph{reduction} of $\frakD_{\fraks}$ at 
degree $\ell$.
Due to the finiteness condition, $\frakD_{\fraks,\ell}$ has only finitely many walls.

Let $\gamma$ be any admissible loop for $\frakD_{\fraks}$.
Let us fix degree $\ell$.
Suppose that $\gamma$ intersects walls  of $\frakD_{\fraks,\ell}$
at $z_1$, \dots, $z_{P-1}$ in this order.
There might be multiple walls with a common normal vector intersected by $\gamma$ at a time.
We distinguish them by allowing the multiplicity $z_a=z_{a+1}$.
Then, the walls crossed at $z_0, \cdots, z_{P-1}$ are parametrized as $\bfw_0,\dots, \bfw_{P-1}$.
By Theorem \ref{thm:CSD1} (b),
each wall $\bfw_a$ has the form
 \begin{align}
 \bfw_a = (\frakd_a, \Psi[t_a n_a]^{s_a \d(t_a n_a)})_{n_a}.
 \end{align}
By the consistency of $\frakD_{\fraks}$,
we have $\frakp_{\gamma, \frakDs}=\rmid$.
Thus, we have a consistency relation around $\gamma$,
\begin{align}
\label{eq:gseq1}
\begin{split}
\frakp_{\gamma, \frakD_{\fraks,\ell}}&=
\Psi[t_{P-1} n_{P-1}]^{\epsilon_{P-1} s_{P-1} \d(t_{P-1} n_{P-1})}\cdots
\Psi[t_1 n_1]^{\epsilon_1 s_1 \d(t_1 n_1)}
\equiv \rmid
\\
&
\hskip230pt
\mod G^{> \ell},
\end{split}
\end{align}
where $\epsilon_a$ is the  intersection sign defined by \eqref{eq:int1}.

For the variable $\bfy=(y_1$\, \dots, $y_r)$,
we call the following formal sum
\begin{align}
f(\bfy) +\sum_{i=1}^r (\log y_i) g_i(\bfy)
\quad
(f(\bfy), g_i(\bfy)\in \bbQ[[Q]])
\end{align}
a \emph{formal (generalized)  Puiseux series in  {$\bfy$} with $\log$ factor}.
Let $\calF^{> \ell}$ be the set of all formal power series $f(\bfy) \in \bbQ[[Q]]$ 
such that  all coefficients of $f(\bfy)$ vanish up to the total order $\ell$.
Let
\begin{align}
\calF^{> \ell}_{\log}:=\calF^{> \ell} + \sum_{i=1}^r (\log y_i) \calF^{> \ell}.
\end{align}

We first prove the reduced DI at level $\ell$
associated with $\gamma$.
\begin{thm}
\label{thm:DI0}
The following identity holds
 as a formal   Puiseux series in  {$\bfy$} with $\log$ factor:
\begin{align}
\label{eq:DI1}
\sum_{a=0}^{P-1}
\epsilon_a s_a
\d(t_a n_a)
\tilde L( y_{z_a}[t_an_a])
\equiv 0
\mod 
\calF^{> \ell}_{\log}.
\end{align}
\end{thm}

Here we present a proof based on the classical mechanical method 
employed in \cite{Gekhtman16} with some modification.
(An alternative proof will be given later by Theorem \ref{thm:DI2}.)
First, let us give an outline of the method.
The basic observation therein and also here, which is originated in \cite{Fock07},  is that the action of $\Psi[n]$ in \eqref{eq:ymut1}
(i.e., a mutation)
is described by a Hamiltonian system
with a log-canonical Poisson bracket
\begin{align}
\label{eq:yyP1}
\{y^n, y^{n'}\}_{\mathrm{P}}:= \{n,n'\} y^n y^{n'},
\quad
\{e_i,e_j\}=\d_i^{-1} b_{ij}
\end{align}
and a Hamiltonian
\begin{align}
\label{eq:Hn1}
\scrH[n]:=\mathrm{Li}_2(-y^n).
\end{align}
Here, we change the normalizations
of \eqref{eq:yyP1} and \eqref{eq:Hn1} 
from the convention in \cite{Gekhtman16}.
Indeed,
\begin{align}
\label{eq:Hy1}
\{\scrH[n], y^{n'}\}_{\mathrm{P}}=-\{n, n'\} y^{n'} \log(1+y^n),
\end{align}
which is the infinitesimal (or log) form of the action of $\Psi[n]^{-1}$
in \eqref{eq:ymut1}.
We also note that the calculation is essentially the same as
\eqref{eq:ymut1}.
The main ingredient of the method in  \cite{Gekhtman16} is the \emph{canonical coordinates} $\bfu=(u_1,\dots,u_r)$ and $\bfp=(p_1,\dots,p_r)$
satisfying
\begin{align}
\{p_i,u_j\}_{\mathrm{P}}=\delta_{ij},
\quad
\{u_i,u_j\}_{\mathrm{P}}=\{p_i,p_j\}_{\mathrm{P}}=0.
\end{align}
Then,
the $y$-variables $\bfy$ are represented as
\begin{align}
\label{eq:ypw1}
y_i = \exp((\d_i^{-1} p_i + w_i )/\sqrt{2}),
\quad
w_i=\sum_{j=1}^r b_{ji} u_j.
\end{align}
In the phase space $\tilde M\simeq \bbR^{2r}$ with the canonical coordinates $(\bfu, \bfp)$,
the point $(\overline\bfu,\overline\bfp)$ moves to  $(\overline\bfu',\overline\bfp')$
under the time-one flow of $\scrH[n]$ as
\begin{align}
\label{eq:ou1}
\overline u_i' &=
\overline u_i - \frac{1}{\sqrt{2}} n_i \d_i^{-1} \log (1+\overline y^{n}),
\\
\label{eq:op1}
\overline p_i' &=
\overline p_i + \frac{1}{\sqrt{2}} 
\biggl(\sum_{j=1}^r n_j b_{ij}\biggr) \log (1+\overline y^{n}),
\end{align}
where $\overline y_i = \exp((\d_i^{-1} \overline p_i + \overline w_i )/\sqrt{2})$.
We concentrate on the subspace (called the \emph{small phase space}) $\tilde M_0$ of $\tilde M$ defined
by $\d_i^{-1}p_i = w_i$ ($i=1,\dots,r$).
Then,
the modified Rogers dilogarithm
\begin{align}
\scrL[n]:=\tilde L ( y^n).
\end{align}
appears as the ``Lagrangian''.
Moreover, it is constant under the above  time-one flow
because $\{\scrH[n],y^n\}_{\mathrm{P}}=0$.
We regard $\scrL[n]$ as a function of $\overline\bfu$
and consider its infinitesimal variation $\delta \scrL[n] $ by $\overline\bfu +\delta\overline\bfu$.
Then, the following formula holds.

\begin{lem}
[{cf.  \cite[Lemma 6.4]{Gekhtman16}}]
\label{lem:dL1}
In the small phase space $\tilde M_0$, we have
\begin{align}
\label{eq:LPu1}
\delta \scrL[n] =
\sum_{i=1}^r \overline p'_i \delta \overline u'_i
-
\sum_{i=1}^r \overline p_i \delta \overline u_i.
\end{align}
\end{lem}

The  formula is identical to the one in \cite[Lemma 6.4]{Gekhtman16}.
However,
the proof therein is  applicable only when $n$ is a $c$-vector
of the cluster pattern for $B$.
So, we give a different line of  a proof.

\begin{proof}
By \eqref{eq:modL1}, the LHS is written as
\begin{align}
\label{eq:dLn1}
\delta
\scrL[n]
=\frac{1}{\sqrt{2}} 
\biggl(\, \sum_{i,j=1}^r n_j b_{ij} \delta \overline u_i\biggr)
\biggl\{
\log(1+\overline y^{n})
-
\frac{  \overline y^{n} \log \overline y^{n}}
{1+\overline y^{n}}
\biggr\}.
\end{align}
By \eqref{eq:ou1}, we have
\begin{align}
\label{eq:ou2}
\delta \overline u_i' &=
\delta \overline u_i - \frac{1}{\sqrt{2}} n_i \d_i^{-1}
\biggl(\, \sum_{j,k=1}^r n_j b_{kj} \delta \overline u_k\biggr)
 \frac{ \overline y^{n}}{1+\overline y^{n}}.
\end{align}
Let us write \eqref{eq:op1} and \eqref{eq:ou2} as
$
\overline p_i' =
\overline p_i + A_i
$
and
$
\delta \overline u_i' =
\delta \overline u_i - B_i
$.
Then, the RHS of \eqref{eq:LPu1} is given by
\begin{align}
\label{eq:pusum1}
\sum_{i=1}^r  A_i  \delta \overline u_i
-
\sum_{i=1}^r    \overline p_i B_i 
-
\sum_{i=1}^r  A_i  B_i.
\end{align}
They are calculated as
\begin{align}
\sum_{i=1}^r  A_i  \delta \overline u_i
&=
\frac{1}{\sqrt{2}}
\biggl(\, \sum_{i, j=1}^r n_j b_{ij}  \delta \overline u_i \biggr) \log (1+\overline y^{n}),
\\
\sum_{i=1}^r    \overline p_i B_i 
&=
\frac{1}{\sqrt{2}}
\biggl(\, \sum_{i, j=1}^r n_j b_{ij}  \delta \overline u_i \biggr) 
\frac{  \overline y^{n} \log \overline y^{n}}
{1+\overline y^{n}},
\\
\sum_{i=1}^r  A_i  B_i
&=
\frac{1}{{2}}
\biggl(\, \sum_{i, j=1}^r n_j b_{ij}  n_i \d_i^{-1}  \biggr) 
\biggl(\, \sum_{i, j=1}^r n_j b_{ij}  \delta \overline u_i \biggr) 
\frac{\overline y^{n} \log (1+ \overline y^{n})}{1+\overline y^{n}}
=0,
\end{align}
where the last equality is due to the skew-symmetry of $\d_i^{-1}b_{ij}$.
Thus, the equality \eqref{eq:LPu1} holds.
\end{proof}

Now we are ready to prove Theorem \ref{thm:DI0}.
\begin{proof}
[Proof of Theorem \ref{thm:DI0}]

We prove \eqref{eq:DI1} 
based on
Lemma  \ref{lem:dL1} .
\par
(a). Assume that $B$ is nonsingular.
Then, the $y$-variables $\bfy$ represented in \eqref{eq:ypw1}
are algebraically independent on the small phase space $\tilde M_0$.

We consider the Hamiltonian system in the small phase space $\tilde M_0$
for the time span $[0,P]$, where $P$ is the one in \eqref{eq:DI1}.
The Hamiltonian for the time span $[a,a+1]$ ($a=0,\dots,P$) is given by
\begin{align}
\scrH_a=
\epsilon_a s_a
\d(t_a n_a)
 \mathrm{Li}_2( -y^{t_a n_a}).
\end{align}
Accordingly, for a trajectory $\alpha=(\overline\bfu(t), \overline\bfp(t))_{t\in [0,j]}$
in the small phase space $\tilde M_0$, we consider the quantity
\begin{align}
\label{eq:Salpha1}
S[\alpha]=\sum_{a=0}^{P-1} \scrL_a,
\quad
\scrL_a  = 
\epsilon_a s_a
\d(t_a n_a)
\tilde L( \overline y(a)^{t_a n_a}),
\end{align}
where 
\begin{align}
\overline y_i(a) =\exp( \sqrt{2} \d_i^{-1} \overline p_i (a) )=\exp\biggl( \sqrt{2}
\sum_{j=1}^r b_{ji} \overline u_j(a)\biggr).
\end{align}
For the $i$th coordinate function $y_i$ at time 0, let
$y'_i$  be the one at time $P$ after the above time development.
Then,
by \eqref{eq:Hy1}, 
it is
 described as a function of $\bfy $
by 
\begin{align}
\label{eq:ypy1}
y'_i = \frakp_{\gamma, \frakD_{\fraks,\ell}}^{-1} (y_i),
\end{align}
where
\begin{align}
\frakp_{\gamma, \frakD_{\fraks,\ell}}^{-1}=
\Psi[t_0 n_0]^{- \epsilon_0 s_0 \d(t_0 n_0)}
\cdots
\Psi[t_{P-1} n_{P-1}]^{- \epsilon_{P-1} s_{P-1} \d(t_{P-1} n_{P-1})}.
\end{align}
Note that the dilogarithm elements act in the opposite order 
along the time development on
the coordinate functions.

For simplicity, suppose that equality for $ \frakp_{\gamma, \frakD_{\fraks,\ell}}$ in 
\eqref{eq:gseq1} is the  exact one \emph{without} modulo $G^{>\ell}$.
Then, by  \eqref{eq:ypy1}, we have
\begin{align}
\label{eq:yiy1}
y'_i = y_i.
\end{align}
It follows form Lemma \ref{lem:dL1} that $\delta S[\alpha]=0$
under any infinitesimal deformation $\alpha + \delta \alpha$.
This implies that $S[\alpha]$ is constant with respect to $\alpha$.
Moreover, it is easy to show that the constant is zero
by taking the limit $ \bfy \rightarrow \bf0$.
On the other hand, 
$S[\alpha]$ can be viewed as a formal  Puiseux 
series in $\bfy=\overline\bfy(0)$ with $\log$ factor.
More explicitly, $\overline y(a)^{t_a n_a}$ in \eqref{eq:Salpha1}
is replaced with
\begin{align}
\label{eq:Psiact1}
\begin{split}
&\
 \Psi[t_0 n_0]^{-\epsilon_0 s_0 \delta (t_0 n_0)}
\cdots
\Psi[t_{a-1} n_{a-1}]^{-\epsilon_{a-1} s_{a-1} \delta (t_{a-1} n_{a-1})}
(y^{t_a n_a})
\\
=&\ 
\frakp_{\gamma_a, \frakD_{\fraks,\ell}}^{-1}(y^{t_a n_a}),
\end{split}
\end{align}
where $\gamma_a$ is the subpath of $\gamma$  from
the base point of $\gamma$ to $z_a$.
Take any admissible path $\gamma_0$ 
from the base point of $\gamma$ to a point in $\mathrm{Int}(\calC^+)$
and apply $ \frakp_{\gamma_0, \frakD_{\fraks,\ell}}$ to $S[\alpha]$.
Then, we obtain the desired formula \eqref{eq:DI1}.

Now we consider the general case in \eqref{eq:gseq1}.
The equality \eqref{eq:yiy1} is replaced with
\begin{align}
\label{eq:yy1}
y'_{i}=
y_{i} (1+ h_i(\bfy))
\end{align}
for some $h_i(\bfy)\in \calF^{>\ell}$.
Then, again by  Lemma \ref{lem:dL1},
under any infinitesimal variation $\alpha + \delta \alpha$,
\begin{gather}
\delta S[\alpha]
=\sum_{i=1}^r (\overline p'_i \delta \overline u'_i -  \overline p_i \delta \overline u_i),
\\
\label{eq:pu1}
\d_i^{-1}\overline p_i= \sum_{j=1}^r b_{ji} \overline u_j= \frac{1}{\sqrt{2}}\log \overline y_i,
\quad
\d_i^{-1}\overline p'_i= \sum_{j=1}^r b_{ji} \overline u'_j= \frac{1}{\sqrt{2}}\log \overline y'_i,
\end{gather}
where we set $\overline p_i = \overline p_i(0)$ and 
$\overline p'_i = \overline p'_i(P)$, etc.
From now on,
we view $\overline p_i$, $\overline p'_i$, $\overline u_i$, $\overline u'_i$
as formal  Puiseux series in $\bfy=\overline\bfy(0)$ with $\log$ factor as before,
and also we omit the bars, for simplicity.
By \eqref{eq:yy1} and \eqref{eq:pu1}, we can write them as
\begin{align}
p'_i- p_i= f_i,
\quad
 u'_j-   u_j=g_i
 \quad
 (f_i,g_i
\in \calF^{>\ell}).
\end{align}
Then, we have
\begin{align}
\begin{split}
\quad \ \delta S[\alpha]
&=\sum_{i=1}^r \biggl((p_i + f_i) \biggl(\delta u_i + \sum_{j=1}^r \frac{\partial g_i }{\partial y_j} 
\frac{\partial y_j }{\partial u_i}\delta u_i
\biggr) -  p_i \delta u_i\biggr)
\\
&=\sum_{i=1}^r ((\log y_i) \tilde f_i + \tilde g_i)\delta u_i
\quad
(\tilde f_i, \tilde g_i\in \calF^{>\ell})
\\
&=\sum_{i=1}^r ((\log y_i) \hat f_i + \hat g_i) \delta y_i
\quad
(\hat f_i, \hat g_i\in \calF^{>\ell-1}).
\end{split}
\end{align}
It follows that
\begin{align}
\frac{\partial S[\alpha]}{\partial y_i} \in \calF_{\log}^{>\ell-1}.
\end{align}
Therefore, we have
\begin{align}
 S[\alpha]\in \calF_{\log}^{>\ell},
 \end{align}
 where the constant term is shown to be zero as before.
 Then, applying  $ \frakp_{\gamma_0, \frakD_{\fraks,\ell}}$
 to $S[\alpha]$, we obtain the desired formula \eqref{eq:DI1}.
 \par
 (b).
When $B$ is singular, 
we have
the problem that $\bfy$ represented in $\eqref{eq:ypw1}$
are not algebraically independent on the small phase space $\tilde M_0$.
To remedy it, we apply the standard principal extension technique \cite{Fomin07,Gross14}.
Namely, let $\tilde N= N \oplus M^{\circ}$. We extend the skew-symmetric bilinear form $\{\cdot,
\cdot\}$ on $N$ to the one on $\tilde N$ as
\begin{align}
\{(n,m),(n',m')\}:=\{n,n'\} + \langle n',m\rangle - \langle n, m'\rangle.
\end{align}
Then, we extend the $y$-representation  \eqref{eq:Xnaction1} of $G$ 
 to the one  (the \emph{principal $y$-representation})
 on $\bbQ[[\tilde \bfy]]$ ($\tilde \bfy=( y_1,\dots,  y_{2r})$)  by
 \begin{align}
\label{eq:Xnaction2}
X_n(\tilde y^{(n',m')})= \{(n,0),(n',m')\} \tilde y^{(n'+n,m')}.
\end{align}
 Accordingly, we extend the canonical coordinates $\bfu$ and $\bfp$ to $\tilde\bfu=(u_1,\dots,u_{2r})$ and $\tilde \bfp=(p_1,\dots,p_{2r})$
 to express $\tilde y_i$
 in the same way as \eqref{eq:ypw1},
 where the matrix $B$ is replaced with  the principally extended matrix of $B$
 \begin{align}
  \tilde B=
 \begin{pmatrix}
B& -I
\\
I & O
 \end{pmatrix}
 .
 \end{align}
 Then, in particular, the subvariables $ \bfy=( y_1,\dots,  y_{r})$ are now algebraically independent
 on the small phase space for $(\tilde \bfu, \tilde \bfp)$.
 After this, the proof of (a) is applicable.
 \end{proof}

Now we take the limit $\ell\rightarrow \infty$.
We parametrize the intersection of $\gamma$ and
the walls in $\frakDs$ as $z_a$ ($a \in J$) by a countable and totally ordered set $J$
so that $\gamma$ crosses $z_a$ earlier than $z_{a'}$ only if $a<a'$.
Then, the consistent relation is presented by the infinite  product
in the increasing order in $J$ from right to left,
\begin{align}
\label{eq:gseq2}
\frakp_{\gamma, \frakD_{\fraks}}=
\prod_{a\in J}^{\leftarrow}
\Psi[t_a n_a]^{\epsilon_a s_a \d(t_a n_a)}
= \rmid.
\end{align}

Thus, we obtain the first main theorem of the paper.
\begin{thm}
\label{thm:DI1}
For any admissible loop $\gamma$ with the consistent relation
\eqref{eq:gseq2},
the following identity holds
as a formal   Puiseux series in  {$\bfy$} with $\log$ factor:
\begin{align}
\label{eq:DI2}
\sum_{a\in J}
\epsilon_a s_a
\d(t_a n_a)
\tilde L( y_{z_a}[t_an_a])
=0.
\end{align}
\end{thm}
\begin{proof}
This is immediately obtained by taking the limit $\ell\rightarrow \infty$
of \eqref{eq:DI1}.
\end{proof}
We call the identity \eqref{eq:DI2}, together with its reduction \eqref{eq:DI1} at degree $\ell$,
the \emph{dilogarithm identity} (DI) associated with a loop $\gamma$ in $\frakDs$.

\section{Construction and reduction of DIs by pentagon identity}

\subsection{Construction and reduction of CSD by pentagon relation}
Let us briefly review the construction of
 a CSD $\frakDs$ 
 satisfying the property in Theorem \ref{thm:CSD1} (b) 
 by \cite{Gross14, Nakanishi22a}.
 
Let $\Gamma$ be a fixed data of rank 2.
We say that a product of dilogarithm elements $\Psi[n]^c$ ($n\in N^+$, $c\in \bbQ_{>0}$)
 is \emph{ordered} (resp. \emph{anti-ordered})
if, for  any adjacent pair $\Psi[n']^{c'}\Psi [n]^{c}$,
$\{n',n\}\leq 0$ (resp. $\{n',n\}\geq 0$) holds.

The following is a key lemma which we use in the construction of the above mentioned CSD.

\begin{prop}[{Ordering lemma, \cite[Prop.\ III.5.4]{Nakanishi22a}}]
\label{prop:rank21}
Let $\fraks$ be a seed for a  fixed data $\Gamma$ of rank 2.
Let 
\begin{align}
\label{3eq:init1}
C^{\mathrm{in}}=
\Psi[t'_j n'_j]^{s'_j  \d(t'_j n'_j)}
\cdots
\Psi[t'_1n'_1]^{s'_1\d(t'_1 n'_1)}
\quad
(n'_a\in N_{\rmpr}^+, \ s'_a,t'_a \in \bbZ_{>0}).
\end{align}
be any finite anti-ordered product.
 Then, $C^{\mathrm{in}}$
equals to a (possibly infinite) ordered product $C^{\mathrm{out}}$ of 
factors  of the same form
\begin{align}
\label{3eq:gpos2}
\Psi[t_a n_a]^{s_a \d(t_a n_a) }
\quad
(n_a\in N_{\rmpr}^+, \ s_a,t_a \in \bbZ_{>0}).
\end{align}
Moreover,
the above relation $C^{\mathrm{in}}=C^{\mathrm{out}}$ is obtained from
a trivial relation $C^{\mathrm{in}}=C^{\mathrm{in}}$ by
 applying  the relations
in Proposition \ref{prop:pent1}
possibly infinitely many times.
\end{prop}
The explicit algorithm of obtaining $C^{\mathrm{out}}$ from $C^{\mathrm{in}}$
is given in \cite[Algorithm III.5.7]{Nakanishi22a}.

We recall important notions for a scattering diagram \cite{Gross14}.

\begin{defn}[Parallel/perpendicular joint]
\par (a).
Let $\frakD$ be a scattering diagram.
For any pair of walls $\bfw_i=(\frakd_i,g_{i})_{n_i}$ ($i=1,2$), the intersection
of their supports $\frakj=\frakd_1\cap \frakd_2$
is called a \emph{joint} of $\frakD$ if $\frakj$
is a cone of codimension 2.
\par (b).
For any joint $\frakj$,
let
\begin{align}
\label{eq:Nj1}
N_{\frakj}:=
\{ n\in N \mid \langle n, z\rangle=0
\ (z\in \frakj)\},
\end{align}
which is the rank 2 sublattice of $N$.
Then, 
a joint $\frakj$ is \emph{parallel}
(resp.\ \emph{perpendicular}) if 
the skew-symmetric form $\{\cdot, \cdot\}$
restricted on $N_{\frakj}$ vanishes (otherwise).
\end{defn}

Based  on Proposition \ref{prop:rank21},
we present  the construction of a CSD $\frakDs$,
which was given  by \cite{Gross14} and modified  with
Proposition \ref{prop:rank21} by \cite{Nakanishi22a}.
The resulting CSD satisfies the property in Theorem \ref{thm:CSD1} (b).

\begin{const}
[{\cite[Appendix C.3]{Gross14}, \cite[Construction III.5.14]{Nakanishi22a}}]
\label{const:CSD1}
We construct scattering diagrams $\frakD_1\subset \frakD_2 \subset \cdots$ as below.
Then, a CSD $\frakDs$ in Theorem \ref{thm:CSD1} (b) 
is given by
\begin{align}
 \frakDs=
 \bigcup_{\l=1}^\infty \frakD_{\l}.
 \end{align}
\par
(1). We start with $\frakD_1=\mathrm{In}_{\fraks}$.
\par
(2). We construct $\frakD_{\l+1}$ from $\frakD_{\l}$ as follows.
For any \emph{perpendicular} joint
$\frakj
$
of $\frakD_{\l}$,
let $N_{\frakj}$ be the one in \eqref{eq:Nj1}.
Let $N_{\frakj}^+ :=N^+\cap N_{\frakj}$
and
 $N_{\frakj,\rmpr}^+ :=N^+_{\rmpr}\cap N_{\frakj}$.
   Then, there exists a unique pair $\tilde{e}_1, \tilde{e}_2\in N_{\frakj,\rmpr}^+$ such that
  $N_{\frakj}^+ \subset \bbQ_{\geq 0}   \tilde{e}_1+  \bbQ_{\geq 0}\tilde{e}_2$
  and $\{ \tilde{e}_2, \tilde{e}_1 \}>0$.
See Figure \ref{fig:2d1}.
Thus, $p^*(n)$ is in the ``second quadrant"
in $M_{\bbR}$ with respect to $\tilde{e}_1^{\perp}$ and $\tilde{e}_2^{\perp}$.
Then, by the construction of walls as described below,
all wall elements in $\frakD_{\l}$ has the form in  \eqref{eq:gt1}.
We say that a product of dilogarithm elements $\Psi[n]^c$ ($n\in N_{\frakj}^+$,
$c\in\bbQ_{>0}$)
 is \emph{ordered} (resp. \emph{anti-ordered})
if, for  any adjacent pair $\Psi[n']^{c'}\Psi [n]^{c}$,
$\{n',n\}\leq 0$ (resp. $\{n',n\}\geq 0$) holds.
We take the wall elements of all walls
which  contain $\frakj$
and lie in the second quadrant with respect to $\tilde{e}_1^{\perp}$ and $\tilde{e}_2^{\perp}$,
and we consider the anti-ordered product of them
\begin{align}
\label{eq:C0}
C^{\mathrm{in}}=\Psi[t'_{j'} n'_{j'}]^{s'_{j'}  \d(t'_{j'} n'_{j'})}
\cdots \Psi[t'_1n'_1]^{s'_1\d(t'_1 n'_1)}
\quad
(\deg(t'_a n'_a) \leq \ell),
\end{align}
where  $n'_a \in N_{\frakj,\rmpr}^+$ and   $s'_a, t'_a\in \bbZ_{>0}$.
Then, we apply Proposition  \ref{prop:rank21}
to $C^{\mathrm{in}}$ modulo $ G^{>\ell+1}$,
and obtain a finite ordered product
\begin{align}
\label{eq:C1}
C^{\mathrm{out}}(\ell)=
\Psi[t_j n_j]^{s_j  \d(t_j n_j)}
\cdots
\Psi[t_1n_1]^{s_1\d(t_1 n_1)}
\quad
(\deg(t_a n_a )\leq \ell+1)
\end{align}
such that 
\begin{align}
C^{\mathrm{in}}\equiv C^{\mathrm{out}}(\ell) \mod G^{>\ell+1}.
\end{align}
By the construction of walls which we describe below,
all factors in \eqref{eq:C1} with $\deg(t_a n_a)\leq \l$
already appear as wall elements in $\frakD_{\l}$.
Now, we add new walls
\begin{align}
\label{3eq:add2}
(\sigma(\frakj, -p^*(n_a)),
\Psi[t_a n_a ]^{s_a \d(t_a n_a)})_{n_a}
\end{align}
 to $\frakD_{\l}$ for all factors  in \eqref{eq:C1}
with $\deg(t_a n_a)= \l+1$.
The cone $\sigma(\frakj, -p^*(n_a))$,
which is generated by $\frakj$ and $-p^*(n_a)$,  is of codimension 1
due to the perpendicular condition $p^*(n_a)\not\in \bbR\frakj$.
Also,
it is outgoing because it is in the ``fourth quadrant" in $M_{\bbR}$
with respect to $\tilde{e}_1^{\perp}$ and $\tilde{e}_2^{\perp}$.
We do the procedure for all  perpendicular joints of $\frakD_{\l}$
to obtain $\frakD_{\l+1}$.
\end{const}

The consistency  around
parallel joints of the above constructed $\frakDs$ is not  obvious, but it was proved in
 \cite[Appendix C.3]{Gross14}.

 \begin{figure}
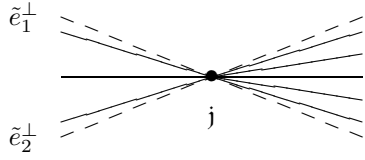

\centering
\leavevmode
\xy
(0,-5)*{\text{\small $\frakj$}};
(-25,8)*{\text{\small $ \tilde{e}_1^{\perp}$}};
(-25,-8)*{\text{\small $ \tilde{e}_2^{\perp}$}};
(0,0)*+{\bullet};
\ar@{--} (-20,8); (20,-8)
\ar@{--} (-20,-8); (20,8)
\ar@{-} (-20,6); (20,-6)
\ar@{-} (-20,0); (20,0)
\ar@{-} (-20,-6); (20,6)
\ar@{-} (0,0); (20,3)
\ar@{-} (0,0); (20,-3)
\endxy
\caption{Walls containing a perpendicular joint $\frakj$.}
\label{fig:2d1}
\end{figure}

The above construction of a CSD immediately
implies the following important result,
which is implicit in \cite[Appendix C.3]{Gross14},
and  described  explicitly in \cite{Nakanishi22a}.
\begin{thm}[{\cite[Theorem III.5.17]{Nakanishi22a}}]
\label{thm:reduce1}
For any admissible loop $\gamma$ for $\frakDs$,
the consistency relation $\frakp_{\gamma, \frakDs}=\rmid$
is reduced  to a trivial one
by applying
 the  relations in Proposition \ref{prop:pent1}
 possibly infinitely many times.
\end{thm}
\begin{proof}
For the reader's convenience,
we present the  proof of  \cite[Theorem III.5.17]{Nakanishi22a}.
Fix $\ell>0$, and consider the reduction $\frakD_{\fraks, \ell}$ at $\ell$.
By the topological reason,
any consistency relation is reduced to the consistency relations for admissible loops
around joints in  $\frakD_{\fraks,\ell}$.
For a parallel joint $\frakj$, which is a trivial case,
the consistency relation has the form
\begin{align}
\label{3eq:triv1}
\begin{split}
&\quad \
\Psi[t'_{j'} n'_{j'}]^{s'_{j'}  \d(t'_{j'} n'_{j'})}
\cdots
\Psi[t'_1n'_1]^{s'_1\d(t'_1 n'_1)}
\\
&
=
\Psi[t'_1n'_1]^{s'_1\d(t'_1 n'_1)}
\cdots
\Psi[t'_{j'} n'_{j'}]^{s'_{j'}  \d(t'_{j'} n'_{j'})},
\end{split}
\end{align}
where the LHS is an anti-ordered product.
By applying \eqref{eq:com1}, it is reduced to a trivial relation.
For a perpendicular joint $\frakj$, 
which is a nontrivial case,
the consistency relation has the form
\begin{align}
\label{eq:const2}
\begin{split}
&\quad \ \Psi[t'_{j'} n'_{j'}]^{s'_{j'}  \d(t'_{j'} n'_{j'})}
\cdots
\Psi[t'_1n'_1]^{s'_1\d(t'_1 n'_1)}
\\
&\equiv
\Psi[t_j n_j]^{s_j  \d(t_j n_j)}
\cdots \Psi[t_1n_1]^{s_1\d(t_1 n_1)}
\mod
G^{>{\l}},
\end{split}
\end{align}
where the left and right  hand sides are the ones in \eqref{eq:C0} and \eqref{eq:C1}, respectively.
The RHS  is obtained from the LHS
by Proposition \ref{prop:rank21}
(more precisely, by Algorithm III.5.7 of \cite{Nakanishi22a}),
which depends only on  the relations \eqref{eq:com1}
and  \eqref{eq:pent1}.
By applying the relations in the reverse way,
the  relation \eqref{eq:const2} is reduced to a trivial one.
\end{proof}

\subsection{Construction and reduction of DIs by pentagon identity}
\label{subsec:constDI1}
Based on Construction \ref{const:CSD1} and Theorem \ref{thm:reduce1},
 we will give parallel results for  DIs,
 where the role of the pentagon relation \eqref{eq:pent1} is played by
 the  pentagon identity \eqref{eq:pentL1} of the Rogers dilogarithm.
 
For this purpose, we reformulate the pentagon identity in
\eqref{eq:pentL1}
into a form which is closer to 
the  pentagon relation \eqref{eq:pent1}.

\begin{lem} Let $n_1, n_2\in N^+$.
 If $\{n_2,n_1\}=c$ $(c\in \bbQ, c\neq 0)$,
the following  pentagon identity holds:
\begin{align}
\label{eq:pentL2}
\begin{split}
\frac{1}{c}
\Psi[n_1 ]^{-1/c} (\tilde L (y^{n_2})) + \frac{1}{c} \tilde L (y^{n_1})
& =
\frac{1}{c} \Psi[n_2 ]^{-1/c} \Psi[n_1+n_2]^{-1/c}(\tilde L( y^{n_1}))
\\
&\quad\
+
\frac{1}{c} \Psi[n_2 ]^{-1/c} (\tilde L( y^{n_1+n_2}))
+
\frac{1}{c}
\tilde L(y^{n_2}).
\end{split}
\end{align}
\end{lem}
\begin{proof}
we have
\begin{align}
\Psi[n_1 ]^{-1/c} (y^{n_2})
&= y^{n_2}  (1+y^{n_1}),
\\
\Psi[n_2 ]^{-1/c} (y^{n_1+n_2})
&= y^{n_1+n_2}  (1+y^{n_2})^{-1},
\\
\begin{split}
\Psi[n_2 ]^{-1/c} \Psi[n_1+n_2]^{-1/c} (y^{n_1})
&=  \Psi[ n_2 ]^{-1/c} (y^{n_1}  (1+y^{n_1+n_2})^{-1})
\\
&=  y^{n_1} (1+y^{n_2})^{-1}
  (1+y^{n_1+n_2} (1+y^{n_2})^{-1})^{-1}\\
&=  y^{n_1} 
  ( 1+y^{n_2}+y^{n_1+n_2})^{-1}.
  \end{split}
\end{align}
Then, the identity \eqref{eq:pentL2}
is obtained from \eqref{eq:pentL1}
by setting $y_1=y^{n_1}$ and $y_2=y^{n_2}$.
\end{proof}

The overall factor $1/c$ was put so that  the identity \eqref{eq:pentL2} perfectly matches  (the log form of) the
corresponding pentagon relation \eqref{eq:pent1}.
Also, observe that action of dilogarithm elements in \eqref{eq:pentL2} 
is parallel to  the one in \eqref{eq:Psiact1}.

Thanks to the correspondence between \eqref{eq:pentL2} and \eqref{eq:pent1},
we  obtain a parallel result to 
Construction \ref{const:CSD1} and Theorem \ref{thm:reduce1} for the DIs.
This is the second main result of the paper.
\begin{thm}
\label{thm:DI2}
For any admissible loop $\gamma$ for $\frakDs$,
the associated DI in \eqref{eq:DI2} 
is constructed from and 
reduced to a trivial one
by applying the pentagon identity in \eqref{eq:pentL1}
possibly infinitely many times.
\end{thm}

\begin{proof}
Fix $\ell>0$, and consider the reduction $\frakD_{\fraks, \ell}$ at $\ell$.
We may concentrate on a sufficiently small loop around a perpendicular joint $\frakj$. 
The constancy  relation is given in \eqref{eq:const2},
where the underlying configuration in $\frakD_{\fraks, \ell}$  is depicted in Figure \ref{fig:sconst2}.
The corresponding DI in \eqref{eq:DI1} has the form
\begin{align}
\label{eq:dd1}
\sum_{a=1}^{j'} s'_a \d(t'_a n'_a) \tilde L (y_{z'_a}[t'_a n'_a])
\equiv
\sum_{a=1}^{j} s_a \d(t_a n_a) \tilde L (y_{z_a}[t_a n_a])
\mod 
\calF^{> \ell}_{\log}.
\end{align}
In view of Figure \ref{fig:sconst2},
this  is rewritten in the form
\begin{align}
\label{eq:dd2}
\begin{split}
&\quad \ 
\sum_{a=1}^{j'} 
\Psi[t'_1 n'_1 ]^{- s'_1\d (t'_1 n'_1)}
\cdots
 \Psi[t'_{a-1} n'_{a-1}]^{- s'_{a-1}  \d (t'_{a-1} n'_{a-1})}
 (
 s'_a \d (t'_a n'_a)
 \tilde L (y^{t'_a n'_a})
 )
 \\
&\equiv
\sum_{a=1}^{j} 
\Psi[t_1 n_1 ]^{- s_1\d (t_1 n_1)}
\cdots
 \Psi[t_{a-1} n_{a-1}]^{- s_{a-1}  \d (t_{a-1} n_{a-1})}
 (
s_a \d (t_a n_a)
 \tilde L (y^{t_a n_a})
 )
 \end{split}
\end{align}
modulo $\calF^{> \ell}_{\log}$.
Along the procedure, say, $\calP$ of obtaining \eqref{eq:C1} from \eqref{eq:C0} by
 Proposition \ref{prop:rank21}
(more precisely, by Algorithm III.5.7 of \cite{Nakanishi22a}),
we do the following procedure:
\begin{itemize}
\item As the initial input, we set $F(y)$ to be the LHS of \eqref{eq:dd2}.
\item If the pentagon relation \eqref{eq:pent1} is applied to the adjacent pair
$\Psi[t' n']^{c'}$ and $\Psi[t n]^c$ 
in the procedure $\calP$,
we apply the pentagon identity \eqref{eq:pentL2}
to the corresponding terms in $F(y)$.
 We also apply  the pentagon relation \eqref{eq:pent1} itself
to the corresponding  pair
$\Psi[t' n']^{c'}$ and $\Psi[t n]^c$ 
in $F(y)$.
(This does not change $F(y)$ as a formal Puiseux series in $\bfy$ with $\log$ factor.)
\item If the commutative relation modulo $G^{>\ell}$
 is applied
to the adjacent pair
$\Psi[t' n']$ and $\Psi[t n]$
in the procedure $\calP$,
 we also apply  it
to the corresponding pair
$\Psi[t' n']^{c'}$ and $\Psi[t n]^c$ 
in $F(y)$.
(This occurs for the relation \eqref{eq:com1}
or  the truncation of the relation \eqref{eq:pent1} with $\deg(n_1+n_2)>\ell$.
In the former case this does not change $F(y)$,
while in the latter case the result equals to $F(y)$  modulo $\calF^{> \ell}_{\log}$.)
\end{itemize}
Then, thanks to the correspondence between 
  \eqref{eq:pentL2}
and
 \eqref{eq:pent1},
 the final result is the RHS of \eqref{eq:dd2}.
 Thus, the equality  \eqref{eq:dd2} is reduced to the trivial one.
Also, by reversing the procedure,
the identity
\eqref{eq:dd2} is obtained from a trivial one.
\end{proof}

\begin{figure}
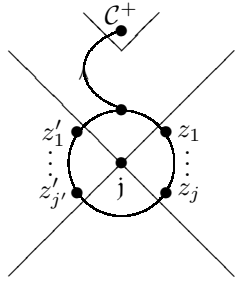

\centering
\leavevmode
\xy
(0,20)*{\text{\small $\calC^+$}};
(-9,4)*{\text{\small $z'_1$}};
(-9,-4)*{\text{\small $z'_{j'}$}};
(9,4)*{\text{\small $z_1$}};
(9,-4)*{\text{\small $z_j$}};
(8.8,1)*{\text{\small $\vdots$}};
(-9.4,1)*{\text{\small $\vdots$}};
(0,-3)*{\text{\small $\frakj$}};
(0,0)*+{\bullet};
(0,7)*+{\bullet};
(0,17.5)*+{\bullet};
(5.9,4.1)*+{\bullet};
(-5.9,4.1)*+{\bullet};
(-5.9,-4.1)*+{\bullet};
(5.9,-4.1)*+{\bullet};
(0,0)*\xycircle<20pt>{};
(1,18)*+{};(1,7)*+{}
   **\crv{(-7,16)&(-7,9)}
          ?>*\dir{};
\ar@{>}(-5,12);(-5,13)
\ar@{-} (-15,-15); (15,15)
\ar@{-} (15,-15); (-15,15)
\ar@{-} (0,15); (-5,20)
\ar@{-} (0,15); (5,20)
\endxy
\caption{Configuration for the consistency relation in
\eqref{eq:const2}.
}
\label{fig:sconst2}
\end{figure}

Note that this also provides an alternative proof of Theorem \ref{thm:DI0}
without the classical mechanical method.
 
 \begin{ex}[Type $B_2$]
 \label{ex:B2}
Let us demonstrate how the procedure in the proof of Theorem \ref{thm:DI2}  works in practice.
Consider the consistency relation  for a CSD of  type $B_2$
obtained by the successive application of the pentagon relation \cite[\S III.2.2]{Nakanishi22a},
\begin{align}
\begin{split}
\begin{bmatrix}
0\\
1
\end{bmatrix}
\left(
\begin{bmatrix}
0\\
1
\end{bmatrix}
\begin{bmatrix}
1\\
0
\end{bmatrix}
\right)
&=
\left(
\begin{bmatrix}
0\\
1
\end{bmatrix}
\begin{bmatrix}
1\\
0
\end{bmatrix}
\right)
\begin{bmatrix}
1\\
1
\end{bmatrix}
\begin{bmatrix}
0\\
1
\end{bmatrix}
\\
&=
\begin{bmatrix}
1\\
0
\end{bmatrix}
\begin{bmatrix}
1\\
1
\end{bmatrix}
\left(
\begin{bmatrix}
0\\
1
\end{bmatrix}
\begin{bmatrix}
1\\
1
\end{bmatrix}
\right)
\begin{bmatrix}
0\\
1
\end{bmatrix}
\\
&=
\begin{bmatrix}
1\\
0
\end{bmatrix}
\begin{bmatrix}
1\\
1
\end{bmatrix}
\begin{bmatrix}
1\\
1
\end{bmatrix}
\begin{bmatrix}
1\\
2
\end{bmatrix}
\begin{bmatrix}
0\\
1
\end{bmatrix}
\begin{bmatrix}
0\\
1
\end{bmatrix}
,
\end{split}
\end{align}
where the pentagon relation \eqref{eq:pent1} is applied to the  pairs in the parentheses.
The associated DI
in the form \eqref{eq:dd2} is obtained as follows:
\begin{align}
\begin{split}
&\
\Psi[\bfe_1]^{-1}\Psi[\bfe_2]^{-1}\tilde L(y^{\bfe_2})
+
(
\Psi[\bfe_1]^{-1}\tilde L (y^{\bfe_2})
+
\tilde L
(y^{\bfe_1})
)
\\
=&\
(
\Psi[\bfe_2]^{-1} \Psi[(1,1)]^{-1} \Psi[\bfe_1]^{-1}\tilde L(y^{\bfe_2})
+
\Psi[\bfe_2]^{-1} \Psi[(1,1)]^{-1} \tilde L (y^{\bfe_1})
)
\\
&\quad
+
\Psi[\bfe_2]^{-1}\tilde L (y^{(1,1)})
+
\tilde L
(y^{\bfe_2})
\\
=&\
\Psi[\bfe_2]^{-1} \Psi[(1,1)]^{-1}  \Psi[\bfe_2]^{-1} \Psi[(1,1)]^{-1}  \tilde L(y^{\bfe_1})
\\
&\quad
+
\Psi[\bfe_2]^{-1} \Psi[(1,1)]^{-1}  \Psi[\bfe_2]^{-1}   \tilde L(y^{(1,1)})
\\
&\quad
+
(\Psi[\bfe_2]^{-1} \Psi[(1,1)]^{-1} \tilde L (y^{\bfe_2})
+
\Psi[\bfe_2]^{-1}\tilde L (y^{(1,1)}))
+
\tilde L
(y^{\bfe_2})
\\
=&\
\Psi[\bfe_2]^{-1}  \Psi[\bfe_2]^{-1}   \Psi[(1,2)]^{-1} \Psi[(1,1)]^{-1} \Psi[(1,1)]^{-1}  \tilde L(y^{\bfe_1})
\\
&\quad
+
  \Psi[\bfe_2]^{-1}  \Psi[\bfe_2]^{-1}   \Psi[(1,2)]^{-1} \Psi[(1,1)]^{-1}  \tilde L(y^{(1,1)})
\\
&\quad
+
   \Psi[\bfe_2]^{-1}   \Psi[\bfe_2]^{-1}   \Psi[(1,2)]^{-1}  \tilde L (y^{(1,1)})
\\
&\quad
+
\Psi[\bfe_2]^{-1}  \Psi[\bfe_2]^{-1}  \tilde L (y^{(1,2)})
+
\Psi[\bfe_2]^{-1}\tilde L (y^{\bfe_2})
+
\tilde L
(y^{\bfe_2}).
\end{split}
\end{align}
For example, in the first equality, the pentagon identity
\eqref{eq:pentL2} is applied to the second and third terms,
while the corresponding pentagon relation  \eqref{eq:pent1} is applied
to the first term.
\end{ex}

Recall that
a sequence of mutations  for a seed of a cluster pattern
is called  a \emph{$\sigma$-period} if it acts  as a permutation $\sigma$ of the indices of the seed \cite{Nakanishi10c}.
Any $\sigma$-period of a cluster pattern corresponds to a  loop $\gamma$ 
contained in the support of the $G$-fan in the corresponding CSD.
Thus,
we have the  following  corollary.
\begin{cor}
\label{cor:DI2}
(a).
For any $\sigma$-period of a cluster pattern,
the corresponding consistency relation $\frakp_{\gamma, \frakD_{\fraks}}=\rmid$
 is reduced to a trivial one
by applying the  commutative and pentagon relations in Proposition \ref{prop:pent1}
possibly  infinitely many times.
\par
(b).
Any DI
associated with a $\sigma$-period of a cluster pattern
in \cite{Nakanishi10c}
is
reduced to a trivial one by applying the pentagon identity in \eqref{eq:pentL1}
possibly  infinitely many times.
\end{cor}

Thus, our DIs have the \emph{infinite reducibility}
in view of the reducibility problem of the DIs to the pentagon relation \cite[\S2.A]{Zagier07}.

\begin{rem}
(a).
The above corollary does \emph{not} imply that 
any $\sigma$-period of a cluster pattern is reduced to a trivial one
by applying the square and pentagon periodicities \emph{in the cluster pattern itself}.
In fact,  Example \ref{ex:B2} is such a case,
where there are no square and pentagon periodicities at all
in the cluster pattern.
Some other examples are also  known \cite{Fomin08, Kim16}.
\par
(b) Even if the corresponding DI for a cluster pattern has only finitely many terms,
it might be  necessary to apply the pentagon relation \emph{infinitely many times}
to reduce it to a trivial one, in general.
\end{rem}

\section{Examples: rank 2 CSDs of affine type}

Let us present examples of DIs for rank 2 CSDs of affine type.
\subsection{Type  $A_1^{(1)}$}
\label{subsec:a11}

Without loosing generality, we may assume that $\{e_2, e_1\}=1$
\cite[\S III.1.5]{Nakanishi22a}. We consider the case $\d_1=\d_2=2$
so that the initial exchange matrix $B$ in \eqref{eq:xyB1} is given by
\begin{align}
B=
\begin{pmatrix}
0 & -2\\
2 & 0
\end{pmatrix}.
\end{align}
Let
$
\bigl[
{
n_1
\atop
n_2
}
\bigr]
$
 denote the dilogarithm element $\Psi[n]$ with $n=n_1 e_1 + n_2 e_2$.
The structure of the CSD is well-known \cite[Example 1.15]{Gross14}, \cite[Theorem 6.1]{Reineke08},
\cite[Theorem 3.4]{Reading18},
and it is represented by the following relation,
\begin{align}
\label{eq:a11rel1}
\begin{bmatrix}
0\\
1
\end{bmatrix}
^2
\begin{bmatrix}
1\\
0
\end{bmatrix}
^2
=
\begin{bmatrix}
1\\
0
\end{bmatrix}
^2
\begin{bmatrix}
2\\
1
\end{bmatrix}
^2
\begin{bmatrix}
3\\
2
\end{bmatrix}
^2
\cdots
\prod_{j=0}^{\infty}
\begin{bmatrix}
2^j\\
2^j
\end{bmatrix}
^
{2^{2-j}}
\cdots
\begin{bmatrix}
2\\
3
\end{bmatrix}
^2
\begin{bmatrix}
1\\
2
\end{bmatrix}
^2
\begin{bmatrix}
0\\
1
\end{bmatrix}
^2
.
\end{align}
See also \cite[\S III.5.7]{Nakanishi22a}, \cite{Matsushita21} for
an alternative derivation of the relation by the pentagon identity.
See Figure \ref{fig:scat1} (a).
Here, the LHS is the  path-ordered  product $\frakp_{\gamma_1, \frakDs}$ along the curve $\gamma_1$,
while
the RHS is the path-ordered product $\frakp_{\gamma_2, \frakDs}$ along the curve $\gamma_2$.
There are infinitely many walls with support $\frakd=\bbR_{\geq 0}(1,-1)\subset (1,1)^{\perp}$.
We have $\d(2^j (1,1))=2^{1-j}$. Thus, the factor $s_a$ is 2 for these walls,
and 1 for the rest.

We only need to consider the loop $\gamma_2^{-1}\circ \gamma_1$,
where the consistency relation is given by the relation \eqref{eq:a11rel1}.
Then, the associated  DI in \eqref{eq:DI1} is written as follows:
\begin{align}
\label{eq:DIa11}
&\tilde L(y_2(1+y_1)^2)+\tilde L(y_1)
=
\tilde L(y_1(1+y_2(1+y_1)^2)^{-2})
+
\Lambda
+
\tilde L(y_2)
,
\\
\begin{split}
&\Lambda= 
\tilde L (y[(2,1)]) 
+ \tilde L (y[(3,2)])
+
\cdots
\\
&\qquad
+
 \sum_{j=0}^{\infty} 2^{1-j} \tilde L (y[2^j(1,1)])
+
\cdots
+\tilde L (y[(2,3)]) 
+ \tilde L (y[(1,2)]).
\end{split}
\end{align}
In the sum $\Lambda$, $y_{z_a}[t_an_a]$ in \eqref{eq:DI1} is unambiguously parametrized by $t_a n_a \in N^+$, so that $z_a$ is omitted.
Also,  the common factor 2 is omitted.
Since the LHS is a finite sum, the RHS converges
as a function of $\bfy\in \bbR^2_{\geq 0}$.

It seems difficult to obtain the explicit expression for $y[t_a n_a]$ in $\Lambda$,
and we do not seek it here.
Instead,  we demonstrate the validity of
  the reduced identity \eqref{eq:DI1} 
for  small degrees $\ell$.

\begin{figure}
\centering
\leavevmode
\begin{xy}
0;/r1.5mm/:,
(0,-15)*{\text{(a) $A_1^{(1)}$}},
(8.7,-10)*{\cdot},
(9.2,-10)*{\cdot},
(10,-8.7)*{\cdot},
(10,-9.2)*{\cdot},
(4,9)*{\gamma_1};
(9,4)*{\gamma_2};
(7,7)*+{\bullet};(-7,-7)*+{\bullet};
(7,7)*+{};(-7,-7)*+{}
   **\crv{(0,6.7)&(-6.7,0)}
          ?>*\dir{>};
       (7,7)*+{};(-7,-7)*+{}
       **\crv{(6.7,0)&(0,-6.7)}
       ?>*\dir{>};
(0,0)="A"
\ar "A"+(0,0); "A"+(10,0)
\ar "A"+(0,0); "A"+(0,10)
\ar@{-} "A"+(0,0); "A"+(-10,0)
\ar@{-} "A"+(0,0); "A"+(0,-10)
\ar@{-} "A"+(0,0); "A"+(5,-10)
\ar@{-} "A"+(0,0); "A"+(6.66,-10)
\ar@{-} "A"+(0,0); "A"+(7.5,-10)
\ar@{-} "A"+(0,0); "A"+(8,-10)
\ar@{-} "A"+(0,0); "A"+(10,-5)
\ar@{-} "A"+(0,0); "A"+(10,-6.66)
\ar@{-} "A"+(0,0); "A"+(10,-7.5)
\ar@{-} "A"+(0,0); "A"+(10,-8)
\ar@{-} "A"+(0,0); "A"+(10,-10)
\end{xy}
\hskip70pt
\begin{xy}
0;/r1.5mm/:,
(0,-15)*{\text{(b) $A_2^{(2)}$}},
(4.1,-10)*{\cdot},
(4.6,-10)*{\cdot},
(5.5,-10)*{\cdot},
(6,-10)*{\cdot},
(0,0)="A"
\ar "A"+(0,0); "A"+(10,0)
\ar "A"+(0,0); "A"+(0,10)
\ar@{-} "A"+(0,0); "A"+(-10,0)
\ar@{-} "A"+(0,0); "A"+(0,-10)
\ar@{-} "A"+(0,0); "A"+(2.5,-10)
\ar@{-} "A"+(0,0); "A"+(3.3,-10)
\ar@{-} "A"+(0,0); "A"+(3.75,-10)
\ar@{-} "A"+(0,0); "A"+(10,-10)
\ar@{-} "A"+(0,0); "A"+(7.5,-10)
\ar@{-} "A"+(0,0); "A"+(6.6,-10)

\ar@{-} "A"+(0,0); "A"+(5,-10)
\end{xy}
\caption{Rank 2 CSDs of affine type.}
\label{fig:scat1}
\end{figure}
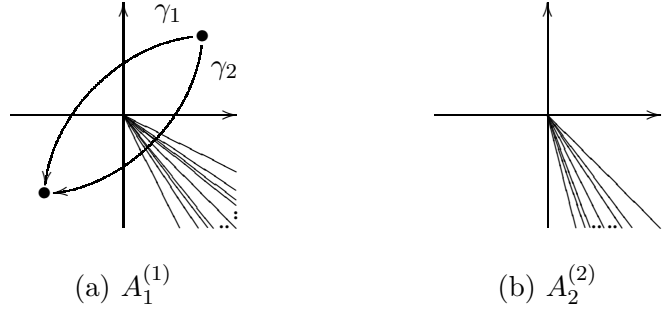

(i) $\ell=1$.
The relation \eqref{eq:a11rel1} is reduced at $\ell=1$ to 
\begin{align}
\begin{bmatrix}
0\\
1
\end{bmatrix}
^2
\begin{bmatrix}
1\\
0
\end{bmatrix}
^{2}
\equiv
\begin{bmatrix}
1\\
0
\end{bmatrix}
^2
\begin{bmatrix}
0\\
1
\end{bmatrix}
^2
\mod
G^{>1}.
\end{align}
We have
\begin{align}
y_2(1+y_1)^2\equiv y_2,
\quad
y_1(1+y_2(1+y_1)^2)^{-2}\equiv y_1
\mod
\calF^{>1}.
\end{align}
The reduced DI at $\ell=1$ is
\begin{align}
\tilde{L}(y_2) + \tilde{L}(y_1) 
\equiv
\tilde{L}(y_1) + \tilde{L}(y_2) 
\mod
\calF^{> 1}_{\log}.
\end{align}
This holds trivially.

(ii) $\ell=2$.
The relation \eqref{eq:a11rel1} is reduced at $\ell=2$ to 
\begin{align}
\begin{bmatrix}
0\\
1
\end{bmatrix}
^2
\begin{bmatrix}
1\\
0
\end{bmatrix}
^{2}
\equiv
\begin{bmatrix}
1\\
0
\end{bmatrix}
^2
\begin{bmatrix}
1\\
1
\end{bmatrix}
^4
\begin{bmatrix}
0\\
1
\end{bmatrix}
^2
\mod
G^{>2}.
\end{align}
We have, modulo $\calF^{>2}$,
\begin{align}
\begin{split}
y_2(1+y_1)^2&\equiv y_2+2y_1y_2,
\\
y_1(1+y_2(1+y_1)^2)^{-2}&\equiv y_1-2y_1y_2,
\\
y[(1,1)]&\equiv y_1y_2
.
\end{split}
\end{align}
Thus,  the reduced DI at $\ell=2$ is
\begin{align}
\tilde{L}( y_2+2y_1y_2 ) + \tilde{L}(y_1) 
\equiv
\tilde{L}(y_1-2y_1y_2) + 2 \tilde{L}(y_1y_2) +\tilde{L}(y_2) 
\mod
\calF^{> 2}_{\log}.
\end{align}
This can be also  verified directly by \eqref{eq:Lexp1}.

(iii) $\ell=3$.
The relation \eqref{eq:a11rel1} is reduced at $\ell=3$ to 
\begin{align}
\begin{bmatrix}
0\\
1
\end{bmatrix}
^2
\begin{bmatrix}
1\\
0
\end{bmatrix}
^{2}
\equiv
\begin{bmatrix}
1\\
0
\end{bmatrix}
^2
\begin{bmatrix}
2\\
1
\end{bmatrix}
^2
\begin{bmatrix}
1\\
1
\end{bmatrix}
^4
\begin{bmatrix}
1\\
2
\end{bmatrix}
^2
\begin{bmatrix}
0\\
1
\end{bmatrix}
^2
\mod
G^{>3}.
\end{align}
We have, modulo $\calF^{>3}$,
\begin{align}
\begin{split}
y_2(1+y_1)^2 &\equiv y_2+2y_1y_2 + y_1^2 y_2,
\\
y_1(1+y_2(1+y_1)^2)^{-2}&\equiv y_1-2y_1y_2-4y_1^2y_2 + 3y_1y_2^2,
\\
y[(1,2)]&\equiv y_1y_2^2. 
\\
y[(1,1)]&\equiv y_1y_2(1+y_{2})^{-2}\equiv y_1y_2- 2y_1 y_2^2.
\\
y[(2,1)]&\equiv y_1^2y_2.
\end{split}
\end{align}
Thus, the reduced DI at $\ell=3$ is
\begin{align}
\begin{split}
&\quad \ \tilde{L}( y_2+2y_1y_2+ y_1^2 y_2 ) + \tilde{L}(y_1) \\
& \equiv
\tilde{L}(y_1-2y_1y_2-4y_1^2y_2 + 3y_1y_2^2) 
+\tilde{L}(y_1^2y_2) 
+
 2 \tilde{L}(y_1y_2- 2y_1 y_2^2) 
 \\
 &\qquad 
+ \tilde{L}(y_1y_2^2) 
 +
 \tilde{L}(y_2) 
\mod
\calF^{> 3}_{\log}.
\end{split}
\end{align}
Again,
this can be also  verified directly by \eqref{eq:Lexp1}.

\subsection{Type  $A_2^{(2)}$}
 We consider the case $\d_1=1$, $\d_2=4$
so that the initial exchange matrix $B$ in \eqref{eq:xyB1} is given by
\begin{align}
B=
\begin{pmatrix}
0 & -1\\
4 & 0
\end{pmatrix}.
\end{align}
The structure of the CSD is well-known \cite{Gross14, Reading18},
and it is represented by the following relation,
\begin{align}
\label{eq:a22rel1}
\begin{split}
&
\begin{bmatrix}
0\\
1
\end{bmatrix}
^4
\begin{bmatrix}
1\\
0
\end{bmatrix}
=
\begin{bmatrix}
1\\
0
\end{bmatrix}
\begin{bmatrix}
1\\
1
\end{bmatrix}
^4
\begin{bmatrix}
3\\
4
\end{bmatrix}
\begin{bmatrix}
2\\
3
\end{bmatrix}
^4
\begin{bmatrix}
5\\
8
\end{bmatrix}
\begin{bmatrix}
3\\
5
\end{bmatrix}
^4
\cdots
\\
&
\hskip60pt
\times
\begin{bmatrix}
1\\
2
\end{bmatrix}
^6
\prod_{j=1}^{\infty}
\begin{bmatrix}
 2^j\\
 2^{j+1}
\end{bmatrix}
^
{2^{2-j}}
\cdots
\begin{bmatrix}
5\\
12
\end{bmatrix}
\begin{bmatrix}
2\\
5
\end{bmatrix}
^4
\begin{bmatrix}
3\\
8
\end{bmatrix}
\begin{bmatrix}
1\\
3
\end{bmatrix}
^4
\begin{bmatrix}
1\\
4
\end{bmatrix}
\begin{bmatrix}
0\\
1
\end{bmatrix}
^4
.
\end{split}
\end{align}
See also \cite[\S III.5.7]{Nakanishi22a}, \cite{Matsushita21} for
an alternative derivation by the pentagon identity.
Here, the canonical paring $\langle n, z\rangle$ for $N\times M_{\bbR}$ is given by
\begin{align}
\langle n, z\rangle
=(n_1,n_2)
\begin{pmatrix}
1 & 0\\
0 & 1/4
\end{pmatrix}
\begin{pmatrix}
z_1\\
z_2
\end{pmatrix}
\end{align}
under the identification $N \simeq \bbZ^2$, $e_i\mapsto \bfe_i$
and $M_{\bbR} \simeq \bbR^2$, $f_i\mapsto \bfe_i$.
There are infinitely many walls with support $\frakd=\bbR_{\geq 0}(1,-2)\subset (1,2)^{\perp}$.
We have $\d(2^j (1,2))=2^{1-j}$. Thus, the factor $s_a$ is 2 or 3 for these walls,
and 1 for the rest.
The associated  DI \eqref{eq:DI1} is written as follows:
\begin{align}
\label{eq:DIa22}
&4 \tilde L(y_2(1+y_1))+\tilde L(y_1)
=
\tilde L(y_1(1+y_2(1+y_1))^{-4})
+
\Lambda
+
4 \tilde L(y_2)
,
\\
\begin{split}
&\Lambda= 
4 \tilde L (y[(1,1)]) 
+ \tilde L (y[(3,4)])
+
\cdots
+
6
\tilde L (y[(1,2)])
\\
&\qquad
+
 \sum_{j=1}^{\infty} 2^{2-j}\tilde L (y[2^j (1, 2)])
+
\cdots
+4\tilde L (y[(1,3)]) 
+ \tilde L (y[(1,4)]).
\end{split}
\end{align}
Since the LHS is a finite sum, the RHS converges
as a function of $\bfy\in \bbR^2_{\geq 0}$.

Let us concentrate on the reduction at $\ell=3$.
The relation \eqref{eq:a22rel1} is reduced at $\ell=3$ to 
\begin{align}
\begin{bmatrix}
0\\
1
\end{bmatrix}
^4
\begin{bmatrix}
1\\
0
\end{bmatrix}
\equiv
\begin{bmatrix}
1\\
0
\end{bmatrix}
\begin{bmatrix}
1\\
1
\end{bmatrix}
^4
\begin{bmatrix}
1\\
2
\end{bmatrix}
^6
\begin{bmatrix}
0\\
1
\end{bmatrix}
^4
\mod
G^{>3}.
\end{align}
We have, modulo $\calF^{>3}$,
\begin{align}
\begin{split}
y_1(1+y_2(1+y_1))^{-4}&\equiv y_1-4y_1y_2-4y_1^2y_2 + 10y_1y_2^2,
\\
y[(1,2)]&\equiv y_1y_2^2. 
\\
y[(1,1)]&\equiv y_1y_2(1+y_{2})^{-4}\equiv y_1y_2- 4y_1 y_2^2.
\end{split}
\end{align}
Thus, the reduced DI at $\ell=3$ is
\begin{align}
\begin{split}
&\quad \ 4 \tilde{L}( y_2+y_1y_2 ) + \tilde{L}(y_1) \\
& \equiv
\tilde{L}(y_1-4y_1y_2-4y_1^2y_2 + 10y_1y_2^2) 
+
 4 \tilde{L}(y_1y_2- 4y_1 y_2^2) 
 \\
 &\qquad 
+ 6 \tilde{L}(y_1y_2^2) 
 +
 4 \tilde{L}(y_2) 
\mod
\calF^{> 3}_{\log}.
\end{split}
\end{align}
Again,
this can be also verified directly by \eqref{eq:Lexp1}.

\bibliographystyle{amsalpha}
\bibliography{../../biblist/biblist.bib}

\end{document}